\documentclass[12pt]{amsart}
\usepackage{fancyhdr}
\usepackage{stmaryrd}
\usepackage{calrsfs}

\usepackage[foot]{amsaddr}
\usepackage{amsmath}
\usepackage[nospace,noadjust]{cite}
\usepackage{amsfonts}
\usepackage{amssymb,enumerate}
\usepackage{amsthm}
\usepackage{cite}
\usepackage{comment}
\usepackage{color}
\usepackage[all]{xy}
\usepackage{hyperref}
\usepackage{lineno}

\newtheorem*{maintheorem*}{Main Theorem}
\newtheorem{theorem}{Theorem}[section]
\newtheorem{prop}[theorem]{Proposition}
\newtheorem{question}[theorem]{Question}
\newtheorem{conj}[theorem]{Conjecture}

\newtheorem{cor}[theorem]{Corollary}

\theoremstyle{definition}
\newtheorem{definition}[theorem]{Definition}
\newtheorem{remark}[theorem]{Remark}
\newtheorem{example}[theorem]{Example}
\numberwithin{equation}{section}

\newcommand{\nn}{\mathbb{N}}
\newcommand{\qq}{\mathbb{Q}}
\newcommand{\rr}{\mathbb{R}}
\newcommand{\zz}{\mathbb{Z}}

\newcommand{\gp}{\text{gp}}

\newcommand{\uu}{\mathcal{U}}

\providecommand\ldb{\llbracket}
\providecommand\rdb{\rrbracket}

\hypersetup{
	pdftitle={Length-Factoriality and Pure Irreducibility},
	colorlinks=true,
	linkcolor=blue,
	citecolor=cyan,
	urlcolor=blue
}


\begin{document}

\newcommand\doctitle{Length-Factoriality and Pure Irreducibility}
	
\begin{center}
    \slshape
    This is an Accepted Manuscript of an article

    published by
    Taylor~\&~Francis in Communications in Algebra on 20 Mar 2023,

    available at:
    \url{https://www.tandfonline.com/doi/10.1080/00927872.2023.2187629}
\end{center}

\title{\doctitle}

\author{Alan Bu$^1$}
\address{$^1$Phillips Exeter Academy\\Exeter, NH 03833}

\author{Joseph Vulakh$^2$}
\address{$^2$Paul Laurence Dunbar High School\\Lexington, KY 40513. joseph@vulakh.us}

\author{Alex Zhao$^3$}
\address{$^3$Lakeside School\\Seattle, WA 98125}

\begin{abstract}
	 An atomic monoid $M$ is called length-factorial if for every non-invertible element $x \in M$, no two distinct factorizations of $x$ into irreducibles have the same length (i.e., number of irreducible factors, counting repetitions). The notion of length-factoriality was introduced by J. Coykendall and W. Smith in 2011 under the term ``other-half-factoriality": they used length-factoriality to provide a characterization of unique factorization domains. In this paper, we study length-factoriality in the more general context of commutative, cancellative monoids. In addition, we study factorization properties related to length-factoriality, namely, the PLS property (recently introduced by Chapman et al.) and bi-length-factoriality in the context of semirings.
\end{abstract}
\medskip

\maketitle

\pagebreak
\section{Introduction}
\label{sec:intro}

A (multiplicative) cancellative commutative monoid $M$ is called atomic if every non-invertible element of $M$ factors into atoms (i.e., irreducible elements), and an integral domain is atomic if its multiplicative monoid is atomic. Every submonoid of a free commutative monoid is atomic, although some elements may have multiple factorizations into atoms. It follows from the definitions that UFDs are atomic domains. In addition, Noetherian domains and Krull domains are well-studied classes of atomic domains that may not be UFDs. The phenomenon of multiple factorizations in atomic monoids and domains has received a great deal of investigation during the last three decades (see \cite{GZ20} and~\cite{AG22} for recent surveys). In this paper, we study atomicity and the phenomenon of multiple factorizations in the settings of monoids, integral domains, and semidomains, putting special emphasis on the property of length-factoriality and on certain special types of atoms, called \emph{pure atoms} by Chapman et al. in~\cite{CCGS21}, that appear in length-factorial monoids.
\smallskip

An atomic monoid $M$ is called a length-factorial monoid provided that no two distinct factorizations of the same non-invertible element of $M$ have the same length (i.e., number of irreducible factors, counting repetitions). The notion of length-factoriality was introduced and first studied in 2011 by Coykendall and Smith in~\cite{CS11} under the term ``other-half-factoriality". In their paper, they proved that an integral domain is length-factorial if and only if it is a UFD. By contrast, it is well known that there are abundant classes of length-factorial monoids that do not satisfy the unique factorization property. More recent studies of length-factoriality have been carried out by Chapman et al. in~\cite{CCGS21} and by Geroldinger and Zhong in~\cite{GZ21}. In Section~\ref{sec:monoids}, we discuss length-factoriality in the context of monoids, proving that every length-factorial monoid $M$ is an FFM (i.e., $M$ is atomic and every non-invertible element of $M$ has only finitely many divisors up to associates). We conclude Section~\ref{sec:monoids} by investigating the structure of LFMs and the number of pure atoms that can appear in an atomic monoid, showing that for any prescribed pair $(m,n) \in \nn^2$, there is an atomic monoid having precisely $m$ purely long atoms and $n$ purely short atoms.
\smallskip

An atom $a_1$ of an atomic monoid $M$ is called purely long if the fact that $a_1 \cdots a_\ell = a'_1 \cdots a'_m$ for some atoms $a_2, \dots, a_\ell, a'_1, \dots, a'_m$ of~$M$ with $a_i$ and $a_j'$ not associates for any $i$, $j$, implies that $\ell > m$. A purely short atom is defined similarly. It turns out that every length-factorial monoid that does not satisfy the unique factorization property has both purely long and purely short atoms \cite[Corollary~4.6]{CCGS21}. On the other hand, it was proved in~\cite[Theorem~6.4]{CCGS21} that an integral domain cannot contain purely long and purely short atoms simultaneously and, in addition, the authors gave examples of Dedekind domains having purely long (resp., purely short) atoms, but not purely short (resp., purely long) atoms. In Section~\ref{sec:monoid domains}, we prove that monoid algebras with rational exponents have neither purely long nor purely short atoms.
\smallskip

Section~\ref{sec:semidomains} is devoted to presenting a still unanswered question about length-factoriality and pure irreducibility in the setting of semidomains. A semidomain is an additive submonoid of an integral domain that is closed under multiplication (unlike subrings, in a semidomain, some elements may not have additive inverses). It is still not known whether $\nn_0$ is the only semidomain with both its additive and its multiplicative monoids being length-factorial. To assist in the study of length-factoriality in semidomains, we extend to the context of semidomains the result of~\cite{CCGS21} that no integral domain has both purely long and purely short atoms.

\bigskip
\section{Background}
\label{sec:background}

In this section, we introduce some terminology and definitions related to the atomicity of cancellative commutative monoids. We use $\mathbb{N}$ and $\mathbb{N}_0$ to denote the sets of positive and non-negative integers, respectively. We denote the set of rational and real numbers by $\mathbb{Q}$ and $\mathbb{R}$, respectively.  
For $a,b \in \zz$, we denote the set of integers between $a$ and $b$, inclusive, as follows: \begin{equation*}
    \llbracket a, b \rrbracket := \{n \in \mathbb{Z} \mid a \leq n \leq b\}.
\end{equation*}

Over the course of this paper, we tacitly assume that the term \textit{monoid} will always mean cancellative commutative monoid. Let $M$ be a monoid. Since $M$ is assumed to be commutative, we will use additive notation, where ``$+$" denotes the operation of $M$ and~$0$ denotes the identity element. We call the invertible elements of $M$ \textit{units}, and the set of all units of $M$ is denoted by $\mathcal{U}(M)$. The monoid~$M$ is called \textit{reduced} if $\mathcal{U}(M) = \{0\}$.  If $M$ is generated by a set $S$, then we write $M = \langle S \rangle$.

The \emph{quotient group} or the \emph{difference group} of $M$ is the set of differences of elements of $M$ (that is, the unique, up to isomorphism, abelian group $\gp(M)$ satisfying that every abelian group containing a copy of $M$ must also contain a copy of $\gp(M)$) \cite[pp. 5--6]{rG84}. The monoid $M$ is called \emph{torsion-free} if, for any $n \in \nn$ and $x, y \in M$, the equality $nx = ny$ implies that $x = y$. One can easily show that $M$ is torsion-free if and only if $\gp(M)$ is a torsion-free abelian group. By definition, the \emph{rank} of $M$ is the rank of $\gp(M)$ as a $\zz$-module, that is, the dimension of the vector space $\qq \otimes_\zz \gp(M)$.

For $x, y \in M$, the element $y$ \emph{divides} $x$ in $M$, denoted $y \mid_M x$, if there exists an element $z \in M$ such that $x = y + z$. We will denote such a divisibility relation by just $y \mid x$ if the monoid is clear from the context. Two elements $x$ and $y$ of $M$ are \emph{associates} if each of $x$ and $y$ divides the other. A subset $S$ of $M$ is \emph{divisor-closed} if $S$ contains all divisors in $M$ of all elements of $S$.

The monoid $M$ is called a \textit{numerical monoid} if $M$ is an additive submonoid of $\mathbb{N}_0$ and $\mathbb{N}_0 \setminus M$ is finite. Numerical monoids have been study for many decades motivated in part by their connection with the Frobenius coin problem. A \textit{Puiseux monoid} is an additive submonoid of $\mathbb{Q}_{\geq 0}$. Puiseux monoids, being a crucial playground for counterexamples in factorization theory and commutative algebra (via monoid domains), have been actively investigated for the last few years (see \cite{GGT21} and references therein).

An element $a \in M \setminus \mathcal{U}(M)$ is called an \textit{atom} (or an \textit{irreducible}) if $a = x + y$ for some $x,y \in M$ implies that either $x$ or $y$ is a unit of $M$. We denote the set of atoms of $M$ by $\mathcal{A}(M)$. An element $p \in M \setminus \mathcal{U}(M)$ is \emph{prime} if $p \mid x + y$ for some $x, y \in M$ implies that $p \mid x$ or $p \mid y$. The monoid $M$ is \textit{atomic} if every non-unit element of $M$ is a sum of atoms, or in other words, $M = \langle \mathcal{A}(M) \rangle$. This definition was introduced by P. Cohn~\cite{pC68} in the context of integral domains. Observe that a monoid $M$ is atomic if and only if the quotient $M_{\text{red}} := M / \mathcal{U}(M)$ is atomic.

For an atomic monoid $M$, let $\mathsf{Z}(M)$ denote the monoid of formal sums of the elements of $\mathcal{A}(M_{\text{red}})$, and let $\pi_M \colon \mathsf{Z}(M) \to M_{\text{red}}$ be the unique monoid homomorphism such that $\pi_M(a) = a$ for every $a \in \mathcal{A}(M_{\text{red}})$. We will write $\pi$ instead of $\pi_M$ when the monoid is clear from context. For each $z = a_1 + a_2 + \cdots + a_{\ell} \in \mathsf{Z}(M)$, the \emph{length} $\ell$ of $z$ is denoted $|z|$. For every $x \in M$, we define \begin{equation*}
    \mathsf{Z}(x) := \mathsf{Z}_M(x) := \pi^{-1}(x + \mathcal{U}(M)),
\end{equation*}
that is, the set of factorizations of $x$, and \begin{equation*}
    \mathsf{L}(x) := \mathsf{L}_M(x) := \{|z| \mid z \in \mathsf{Z}(x)\}.
\end{equation*}

A pair $(z_1, z_2)$ of factorizations in $\mathsf{Z}(M)$ of the same element of $M$ is called a \emph{factorization relation}, and is said to be \emph{irredundant} if no atom appears in both $z_1$ and $z_2$. In addition, the factorization relation is said to be \emph{balanced} if $|z_1| = |z_2|$, and \emph{unbalanced} otherwise. A factorization relation $(z_1, z_2)$ is called a \emph{master factorization relation} if every irredundant, unbalanced factorization relation of $M$ is of the form $(nz_1, nz_2)$ or $(nz_2, nz_1)$, with $n \in \nn$. Master factorization relations were introduced in the context of integral domains in~\cite{CS11} to study length-factoriality, and the same notion was later generalized to monoids in~\cite{CCGS21}.

If $M$ is atomic and $|\mathsf{Z}(x)| < \infty$ for all $x \in M$, then $M$ is called a \textit{finite factorization monoid}, or an FFM~\cite{AAZ90}. If every element of $M$ has only a finite number of atomic divisors up to associates, then $M$ is called an \emph{idf-monoid}. By \cite[Theorem~5.1]{AAZ90}, the multiplicative monoid of an integral domain is an FFM if and only if it is an atomic idf-monoid; their proof easily extends to all monoids. If $M$ is atomic and $|\mathsf{L}(x)| < \infty$ for all $x \in M$, then the monoid $M$ is called a \textit{bounded factorization monoid}, or a BFM~\cite{AAZ90}. If $|\mathsf{Z}(x)| = 1$ for all $x \in M$, then the monoid $M$ is called a \textit{unique factorization monoid}, or a UFM. If $|\mathsf{L}(x)| = 1$ for all $x \in M$, then the monoid $M$ is called a \textit{half-factorial monoid}, or HFM~\cite{aZ76}. It follows directly from the definitions that every UFM is both an FFM and and an HFM, and also that if a monoid is either an FFM or an HFM, then it must be a BFM.

The monoid $M$ is called a \emph{length-factorial monoid}, or an LFM, if it is atomic and for every nonunit $x$ of $M$ and factorizations $z_1, z_2 \in \mathsf{Z}(x)$, the equality $|z_1| = |z_2|$ implies that $z_1 = z_2$. It follows directly from the definitions that every UFM is an LFM. A monoid is a \emph{proper} LFM if it is an LFM, but not a UFM. An atom $a$ of $M$ is said to be \emph{purely long} (resp., \emph{purely short}) if it is not prime and for every irredundant factorization relation $(z_1, z_2)$, the inclusion $a \in z_1$ implies $|z_1| > |z_2|$ (resp., $|z_1| < |z_2|$). Following \cite{CCGS21}, we let $\mathcal{L}(M)$ (resp., $\mathcal{S}(M)$) denote the set of purely long (resp., purely short) atoms of $M_{\text{red}}$. Elements of $\mathcal{L}(M) \cup \mathcal{S}(M)$ are called \emph{pure} atoms. The monoid $M$ satisfies the \emph{PLS property} and is a \emph{PLS monoid}, or a PLSM, if both sets $\mathcal{L}(M)$ and $\mathcal{S}(M)$ are nonempty. Length-factoriality and pure atoms were recently studied in~\cite{CCGS21}. We proceed to list some of the established results for easy reference.

\begin{prop}[{\cite[Corollary~3.2]{CCGS21}}]
    An atomic monoid $M$ is a proper LFM if and only if it admits an unbalanced master factorization relation $(z_1, z_2)$, in which case the only master factorizations are $(z_1, z_2)$ and $(z_2, z_1)$.
    \label{prop:master_relation}
\end{prop}

\begin{prop}[{\cite[Proposition~4.3]{CCGS21}}]
    Let $M$ be an atomic monoid, let $a \in \mathcal{A}(M)$ be a pure atom, and let $(z_1, z_2)$ be an irredundant factorization relation.  Then $a$ is in $z_1$ or~$z_2$ if and only if $(z_1, z_2)$ is unbalanced.
    \label{prop:pure_in_unbalanced}
\end{prop}

When discussing factorization properties of an integral domain $R$, we refer to the factorization properties of its multiplicative monoid of nonzero elements, denoted here by $R^\bullet$. An integral domain is an \emph{atomic domain} (resp., \emph{FFD}, \emph{idf-domain}, \emph{BFD}, \emph{UFD}, \emph{HFD}, \emph{LFD}) if its multiplicative monoid is an atomic monoid (resp., FFM, idf-monoid, BFM, UFM, HFM, LFM).

It is sometimes useful to construct an integral domain from a given monoid. One important construction of this nature is a \emph{monoid domain}. For an integral domain $R$ and a torsion-free monoid $M$ written additively, the ring $R[X; M]$, also denoted by $R[M]$, is defined as the ring of polynomial expressions of $X$ with coefficients in $R$ and exponents in $M$, with addition and multiplication defined as for standard polynomial rings. The ring $R[M]$ is an integral domain by~\cite[Theorem~8.1]{rG84}. A general reference for commutative monoid domains is~\cite{rG84}. When the integral domain $R$ is a field, the domain $R[M]$ is called a \emph{monoid algebra}.

\bigskip
\section{The Setting of Monoids}
\label{sec:monoids}

\subsection{Connection to the Finite Factorization Property}

Since Anderson, Anderson, and Zafrullah introduced the bounded and finite factorization properties in~\cite{AAZ90}, the chain of implications $UFM \ \Rightarrow \ FFM \ \Rightarrow BFM \Rightarrow \ ACCP$ has been studied in several classes of atomic monoids and domains, and it has been used to better understand the deviation of monoids and domains from satisfying the unique factorization property. It turns out that the property of being length-factorial fits nicely in the same chain of implications as follows:
\[
	\textbf{UFM} \ \Rightarrow \ \textbf{LFM} \ \Rightarrow \ \textbf{FFM}. 
\]
One can immediately see from the definitions that every UFM is an LFM. Let us proceed to prove that every LFM is an FFM.

\begin{prop}
    Every LFM is an FFM.
\end{prop}

\begin{proof}
	Let $M$ be an LFM. Observe that $M$ is an LFM (resp., an FFM) if and only if $M_{\text{red}}$ is an LFM (resp., an FFM). Therefore we can assume, without loss of generality, that $M$ is reduced. Fix $x \in M$, and let us show that the set $A_x$ consisting of all atoms dividing $x$ in $M$ is finite. Because $A_x$ is finite when $\mathsf{Z}(x)$ is a singleton, we can assume that $|\mathsf{Z}(x)| \ge 2$. Since $M$ is an LFM, $|\mathsf{L}(x)| \ge 2$. Set $\ell_1 = \min \mathsf{L}(x)$ and $\ell_2 = \min (\mathsf{L}(x) \setminus \{\ell_1\})$, and take $z_1, z_2 \in \mathsf{Z}(x)$ such that $|z_1| = \ell_1$ and $|z_2| = \ell_2$. We show that every atom in $A_x$ must appear in either $z_1$ or $z_2$. Take $a \in \mathcal{A}(M)$ such that $a \mid_M x$, and pick $z_3 \in \mathsf{Z}(x)$ such that $a$ appears in $z_3$. If $\ell_3 := |z_3| \in \{\ell_1, \ell_2\}$, then the fact that $M$ is an LFM ensures that $z_3 \in \{z_1, z_2\}$, and so $a$ appears in either $z_1$ or $z_2$. Assume, on the other hand, that $\ell_3 \notin \{\ell_1, \ell_2\}$, so that $\ell_1 < \ell_2 < \ell_3$. Observe now that $z_2^{\ell_3 - \ell_1}$ and $z_1^{\ell_3 - \ell_2} z_3^{\ell_2 - \ell_1}$ are factorizations in $\mathsf{Z}(x^{\ell_3 - \ell_1})$ having the same length, namely, $\ell_2(\ell_3 - \ell_1)$. So since $M$ is an LFM, $z_2^{\ell_3 - \ell_1} = z_1^{\ell_3 - \ell_2} z_3^{\ell_2 - \ell_1}$. Thus, $a$ appears in $z_2$. Therefore the only atoms dividing $x$ in $M$ are those that appear in either $z_1$ or $z_2$. Hence $M$ is an atomic idf-monoid, and therefore an FFM.
\end{proof}

According to \cite[Corollary~4.6]{CCGS21}, every proper LFM is a PLSM. Therefore a proper LFM is both an FFM and a PLSM. As the following example illustrates, there are monoids
that are FFMs and PLSMs simultaneously, but still fail to be LFMs.

\begin{example}\label{ex:ffm_plsm_not_lfm}
    Consider the numerical monoid $M := \nn_0 \setminus \{1\}$ and the additive submonoid $N := \{ (0, 0) \} \cup (\nn_0 \times \nn)$ of $\zz^2$. Both monoids are reduced atomic monoids, with $\mathcal{A}(M) = \{2, 3\}$ and $\mathcal{A}(N) = \{(n, 1) \mid n \in \nn_0\}$. Therefore the monoid $M \times N$ is atomic, and
	\begin{equation*}
		\mathcal{A}(M \times N) = \{a_1, a_2\} \cup A,
	\end{equation*}
        where $a_1 := (2,(0,0))$, $a_2 := (3,(0,0))$, and $A := \{ (0,a) \mid a \in \mathcal{A}(N) \}$. We claim that $a_1 \in \mathcal{L}(M \times N)$ and $a_2 \in \mathcal{S}(M \times N)$. To see this, suppose that $(z, z')$ is an irredundant unbalanced factorization relation of $M \times N$, where $z = c_1 a_1 + c_2 a_2 + \sum_{a \in A} c_a a$ and $z' = c'_1 a_1 + c'_2 a_2 + \sum_{a \in A} c'_a a$ for some $c_1, c'_1, c_2, c'_2, c_a, c'_a \in \nn_0$ with $c_a = c'_a = 0$ for all but finitely many $a \in A$. Assume, without loss of generality, that $|z| < |z'|$. Now observe that $N$ is an HFM, and so $\sum_{a \in A} c_a a = \sum_{a \in A} c'_a a$ implies that $\sum_{a \in A} c_a = \sum_{a' \in A} c'_a$. Thus, $c_1 + c_2 < c'_1 + c'_2$. Then the irredundancy of $(z_1, z_2)$, along with the equality $2c_1 + 3 c_2 = 2c'_1 + 3 c'_2$, ensures that $c_1 = c'_2 = 0$ and $0 \notin \{c'_1, c_2\}$. As a consequence, $a_1$ is a purely long atom and $a_2$ is a purely short atom (indeed, we have proved that $\mathcal{L}(M \times N) = \{a_1\}$ and $\mathcal{S}(M \times N) = \{a_2\}$). Hence $M \times N$ is a PLSM.
	The monoid $M \times N$ is an FFM because it is the direct product of two FFMs. Finally, we observe that $M \times N$ is not an LFM because the element $(0, (2, 2))$ can be factored as $(0, (0, 1)) + (0, (2, 1))$ or as $(0, (1, 1)) + (0, (1, 1))$, which have the same length.
\end{example}

Neither of the conditions of being an FFM or being a PLSM implies the other. Indeed, an atomic domain cannot have purely long and purely short atoms simultaneously \cite[Theorem~6.4]{CCGS21}, and so the multiplicative monoid of any finite factorization domain is an example of an FFM that is not a PLSM. In addition, the following example provides a rank-1 FFM that is not a PLSM.

\begin{example}\label{ex:ffm_not_plsm}
    For each $n \in \nn$ with $n \geq 3$, the monoid $N = \langle n, n + 1, n + 2 \rangle$ is a torsion-free rank-1 monoid. Since $n \geq 3$, we have $\mathcal{A}(N) = \{ n, n + 1, n + 2 \}$. Observe that $N$ is an FFM. On the other hand, it follows from~\cite[Proposition~5.7]{CCGS21} that $N$ is not a PLSM.
\end{example}

In the following example, we provide a PLSM that is not an FFM.

\begin{example}\label{ex:plsm_not_ffm}
	Consider the numerical monoid $M := \nn_0 \setminus \{1\}$ and the additive submonoid $N := \{(0,0)\} \cup (\zz \times \nn)$ of $\zz^2$. Note that $M$ and $N$ are reduced atomic monoids with $\mathcal{A}(M) = \{2,3\}$ and $\mathcal{A}(N) = \{(n,1) \mid n \in \zz\}$. Therefore the monoid $M \times N$ is atomic, and
	\[
		\mathcal{A}(M \times N) = \{a_1, a_2\} \cup A,
	\]
        where $a_1 := (2,(0,0))$, $a_2 := (3,(0,0))$, and $A := \{ (0,a) \mid a \in \mathcal{A}(N) \}$. As in Example~\ref{ex:ffm_plsm_not_lfm}, $M \times N$ is a PLSM. However, $M \times N$ is not an FFM because every atom in the infinite set $A$ divides the element $(0, (0,2))$ in $M \times N$.
\end{example}

\smallskip
\subsection{Inheritance}
The notion of hereditary atomicity was introduced and studied in~\cite{CGH21}, where an integral domain is defined to be hereditarily atomic if all of its subrings are atomic. In this section, we extend this notion to other properties.

We say that a monoid $M$ is a \emph{hereditary LFM} if every submonoid of $M$ is an LFM. It turns out that the class of monoids that are hereditary LFMs is rather small.

\begin{prop}
    A monoid $M$ is a hereditary LFM if and only if $M$ is a torsion abelian group.
\end{prop}

\begin{proof}
    For the direct implication, suppose for the sake of contradiction that some element $x$ of $M$ has infinite order. By~\cite[Proposition~5.7]{CCGS21}, the submonoid $\langle 3x, 4x, 5x \rangle$ is not an LFM, so $M$ is not a hereditary LFM.
    
    The reverse implication follows from the fact that any submonoid of a torsion abelian group is a torsion abelian group, and that any abelian group is vacuously an LFM.
\end{proof}

We define a monoid $M$ to be a \emph{hereditary FFM} if every submonoid of $M$ is an FFM. We have seen in Example~\ref{ex:ffm_not_plsm} that for each $n \in \nn$ with $n \geq 3$, the numerical monoid $\langle n, n + 1, n + 2\rangle$ is an example of an FFM that is not a PLSM. From the fact that every submonoid of $N$ is a submonoid of $\nn$, we conclude that $N$ is a hereditary FFM that is not a PLSM.

As the following proposition indicates, the hereditary version of the PLS property is a vacuous property.

\begin{prop}
    Every monoid contains a nontrivial submonoid with no pure atoms.
\end{prop}

\begin{proof}
    Let $M$ be a monoid. If some element $x$ of $M$ has infinite order, then the submonoid $\langle 3x, 4x, 5x \rangle$ of $M$ has no pure atoms by~\cite[Corollary~5.8]{CCGS21}. Assume, therefore, that every element of $M$ has finite order. Then $M$ is a torsion abelian group, so $M$ does not contain any atoms, and thus $M$ has no pure atoms.
\end{proof}

\smallskip
\subsection{Attainability of Master Factorization Relations}

We extend Proposition~\ref{prop:master_relation} by showing that any valid factorization relation written in lowest terms can be used as a master factorization relation.

\begin{theorem}\label{thm:mf_classify}
    Let $m, n$ be nonnegative integers, and let $\{a_i\}_{i = 1}^m$ and $\{b_i\}_{i = 1}^n$ be two sequences of positive integers such that no integer greater than $1$ divides every $a_i$ and $b_i$, if $m = 1$ then $a_1 \neq 1$, if $n = 1$ then $b_1 \neq 1$, and $\sum_{i = 1}^m a_i > \sum_{i = 1}^n b_i$. Then there exists an LFM with purely long atoms $\{\alpha_i\}_{i = 1}^m$ and purely short atoms $\{\beta_i\}_{i = 1}^n$ having \begin{equation}\label{eq:mf_classify}
        \sum\limits_{i = 1}^m a_i \alpha_i = \sum\limits_{i = 1}^n b_i \beta_i
    \end{equation}
    as an unbalanced master factorization relation.
\end{theorem}

\begin{proof}
    Let $\{\alpha_i\}_{i = 1}^m$ and $\{\beta_i\}_{i = 1}^{n - 1}$ be linearly independent indeterminates, and set \begin{equation}\label{eq:beta_n_def}
        \beta_n = \frac1{b_n} \left( \sum\limits_{i = 1}^m a_i \alpha_i - \sum\limits_{i = 1}^{n - 1} b_i \beta_i \right).
    \end{equation}
    Let $M$ be the additive monoid generated by the set $\{\alpha_i, \beta_j \mid i \in \llbracket 1, m \rrbracket \text{ and } j \in \llbracket 1, n \rrbracket \}$. We now show that any factorization relation involving the $\alpha_i$ and $\beta_i$ is a multiple of~(\ref{eq:mf_classify}). Let \begin{equation*}
        \sum\limits_{i = 1}^m a'_i \alpha_i = \sum\limits_{i = 1}^n b'_i \beta_i
    \end{equation*}
    in $\gp(M)$, with all $a'_i$ and $b'_i$ integers (possibly negative). Multiplying by $b_n$ and subtracting the equality~(\ref{eq:mf_classify}) multiplied by $b'_n$, we obtain \begin{equation*}
        \sum\limits_{i = 1}^m (b_n a'_i - b'_n a_i) \alpha_i = \sum\limits_{i = 1}^n (b_n b'_i - b'_n b_i) \beta_i
    \end{equation*}
    in $\gp(M)$.
    The coefficient of $\beta_n$ in the resulting relation is $0$, so by linear independence, $b_n a'_i = b'_n a_i$ and $b_n b'_i = b'_n b_i$ for all $i$. Setting $c = \frac{b'_n}{b_n}$, we obtain $a'_i = ca_i$ and $b'_i = cb_i$. Therefore any linear relation involving the $\alpha_i$ and $\beta_i$ is a multiple of~(\ref{eq:mf_classify}), as claimed.
    
    We now show that all $\alpha_i$ and $\beta_i$ are atoms. Suppose for the sake of contradiction that some $\alpha_k$ can be written as a sum of the other $\alpha_i$ and $\beta_i$. The resulting factorization must be a multiple of~(\ref{eq:mf_classify}), which is only possible if $m = 1$. But when $m = 1$, the condition that $a_1 \neq 1$ prevents $\alpha_1$ from being reducible. The same argument shows that every $\beta_i$ is an atom.

    It follows from what has been proved that~(\ref{eq:mf_classify}) is a master factorization relation. Therefore, by Proposition~\ref{prop:master_relation}, $M$ is an LFM, and the condition that $\sum_{i = 1}^m a_i > \sum_{j = 1}^n b_j$ ensures that each $\alpha_i$ is a purely long atom and that each $\beta_j$ is a purely short atom.
\end{proof}

By~\cite[Theorem~6.4]{CCGS21}, an integral domain cannot have both purely long irreducibles and purely short irreducibles. Theorem~\ref{thm:mf_classify} leads to a contrasting result for monoids.

\begin{cor}\label{cor:pls_m_n}
	For every pair $(m,n) \in \nn^2$, there exists a monoid $M$ with $|\mathcal{L}(M)| = m$ and $|\mathcal{S}(M)| = n$.
\end{cor}

\begin{remark}
    The atoms in the proof of Theorem~\ref{thm:mf_classify} can be taken to be linearly independent positive real numbers, showing that the statement holds for submonoids of the nonnegative real numbers.
\end{remark}

\bigskip
\section{Monoid Domains}
\label{sec:monoid domains}

As proved in~\cite[Theorem~6.4]{CCGS21}, an integral domain cannot contain purely long and purely short irreducibles simultaneously. However, in the same paper, the authors constructed a Dedekind domain having purely long (resp., purely short) irreducibles, but not purely short (resp., purely long) irreducibles.
As we proceed to prove, the same type of construction is not possible in the setting of monoid algebras with rational exponents. The atomicity of monoid algebras with rational exponents has been recently studied in~\cite{fG22}.

For a positive rational number $r$, let $\mathsf{n}(r)$ and $\mathsf{d}(r)$ be the numerator and denominator, respectively, of $r$ when it is written in lowest terms.

\begin{theorem}\label{thm:monalg_no_pls}
	Let $F$ be a field, and let $M$ be an atomic Puiseux monoid. Then $F[M]$ contains no pure irreducibles.
\end{theorem}

\begin{proof}
	Set $R = F[M]$. Assume first that $|\mathcal{A}(M)| = 1$. Since $M$ is atomic, it is clear that $M \cong \nn_0$, and so $R$ is isomorphic to the polynomial ring $F[X]$, which is a UFD. As a result, both sets $\mathcal{L}(M)$ and $\mathcal{S}(M)$ are empty. Assume for the rest of the proof that $|\mathcal{A}(M)| \ge 2$. We split the rest of the proof into the following two cases.
	\smallskip
	
        \noindent \textit{Case 1:} $|\mathcal{A}(M)| > 2$. Let $a_1, a_2$, and $a_3$ be three distinct atoms of $M$. Now observe that after setting $z_1 := (X^{a_1})^{\mathsf{n}(a_2) \mathsf{d}(a_1)}$ and $z_2 := (X^{a_2})^{\mathsf{n}(a_1) \mathsf{d}(a_2)}$, we obtain that $\pi_R(z_1) = X^{\mathsf{n}(a_1) \mathsf{n}(a_2)} = \pi_R(z_2)$. Therefore $(z_1, z_2)$ is an irredundant unbalanced factorization relation of $R$, and so $\mathcal{L}(R) \cup \mathcal{S}(R) \subseteq \{X^{a_1}, X^{a_2} \}$. We can repeat the same argument replacing the pair $(a_1, a_2)$ by the pairs $(a_1, a_3)$ and $(a_2, a_3)$ to obtain that $\mathcal{L}(R) \cup \mathcal{S}(R)$ is empty by Proposition~\ref{prop:pure_in_unbalanced}.
	\smallskip
	
        \noindent \textit{Case 2:} $|\mathcal{A}(M)| = 2$. Since multiplication by a positive rational number induces an isomorphism of Puiseux monoids, we can assume that $M$ is an additive submonoid of $\nn_0$ with $\mathcal{A}(M) = \{a,b\}$ for some $a,b \in \nn_{\ge 2}$ with $a < b$ and $\gcd(a, b) = 1$. Because $\{q X^n \mid q \in F^\bullet \text{ and } n \in M\}$ is a divisor-closed multiplicative submonoid of $R^\bullet$ with reduced monoid isomorphic to $M$, the only irreducible monomials in $R$ are $X^a$ and $X^b$ up to associates. Therefore the monomial $X^{ab}$ has only two factorizations in $R$, namely, $X^{ab} = \big( X^a\big)^b = \big( X^b \big)^a$. As a result, by Proposition~\ref{prop:pure_in_unbalanced}, if $\mathcal{L}(M)$ (resp., $\mathcal{S}(M)$) is nonempty, then $\mathcal{L}(M) = \big\{ qX^a \mid q \in F^\bullet \big\}$ (resp., $\mathcal{S}(M) = \big\{ qX^b \mid q \in F^\bullet \big\}$). Let $p$ and $r$ be the minimum positive integers such that $b \mid pa - 1$ and $a \mid rb -1$, respectively. Now write $pa - Qb = 1$ and $rb - Sa = 1$ for some $Q,S \in \nn$.
	\smallskip
	
	\noindent \emph{Claim.} The elements $X^{pa} - X^{Qb}$ and $X^{rb} - X^{Sa}$ are irreducibles in $R$.
	\smallskip
	
	\noindent \emph{Proof of Claim.}  If $X^{pa} - X^{Qb}$ factors, then since $X^{pa} - X^{Qb}$ factors in $F[X]$ as $X^{Qb}(X - 1)$ and $F[X]$ is a UFD, $X^{pa} - X^{Qb}$ is divisible by a monomial irreducible of $R$. So suppose for the sake of contradiction that $X^a$ or $X^b$ divides $X^{pa} - X^{Qb}$. Then either $Qb - a$ or $pa - b$ is in $M$. Both cases contradict the minimality of $p$. Indeed, if $Qb - a \in M$, we may write $Qb = p'a + qb$, with $p' > 0$, so that $(p - p')a - qb = 1$. Similarly, if $pa - b \in M$, we may write $pa = p'a + qb$, with $p' < p$, so that $p'a - (Q - q)b = 1$. Therefore $X^{pa} - X^{Qb}$ is irreducible.
        Similarly, $X^{rb} - X^{Sa}$ is irreducible.
	\smallskip
	
	Set $a_1(X) := X^{rb} - X^{Sa}$ and $a_2(X) := X^{pa} -
        X^{Qb}$. By the claim, $a_1(X), a_2(X) \in
        \mathcal{A}(R)$. In addition, after setting $c := |Sa - Qb|$, we
        can write either $a_1(X) = X^c a_2(X)$ or $a_2(X) = X^c a_1(X)$ in $F[X]$. We will demonstrate the case when $Sa - Qb > 0$, with the other case being analogous. Consider the
        element $f := X^{ca} a_1(X)^{b-a}$ of $R$. Observe that $z_1 : = (X^a)^c a_1(X)^{b-a}$ is a factorization of $f$ with length $c + b - a$. On the other hand, setting $z_2 := (X^b)^c a_2(X)^{b-a}$, we obtain
	\[
		\pi(z_2) = (X^b)^c a_2(X)^{b-a} = X^{ca} (X^{c(b-a)}a_2(X)^{b-a}) = X^{ca} a_1(X)^{b-a} = f,
	\]
	so $z_2$ is also a factorization of $f$ with $|z_2| = |z_1|$. This implies that neither of the monomials $q X^a$ and $q X^b$ are pure irreducibles for any $q \in F^\bullet$. As a result, both sets $\mathcal{L}(M)$ and $\mathcal{S}(M)$ are empty, which concludes our proof.
\end{proof}

As a consequence of Theorem~\ref{thm:monalg_no_pls}, we obtain that the PLS property does not ascend, in general, from monoids to their corresponding monoid algebras.

\bigskip
\section{Semidomains}
\label{sec:semidomains}

In this last section, we state an open problem related to PLSMs and FFMs, which takes place in the context of semidomains.

\subsection{What Is a Semidomain?}

In the scope of this paper, a \emph{commutative semiring} $S$ is a nonempty set with two binary operations, which are denoted by ``$+$" and ``$\cdot$" and called \emph{addition} and \emph{multiplication}, respectively, such that the following statements hold.
\begin{itemize}
	\item $(S,+)$ is a monoid with identity $0$.
	\smallskip
	
	
	\item $(S, \cdot)$ is a commutative semigroup with identity $1$.
	\smallskip
	
	\item Multiplication distributes over addition.
	\smallskip

\end{itemize}

Let $S$ be a commutative semiring. For any $a,b \in S$, we write $a b$ instead of $a \cdot b$. Notice that for all $a \in S$, the equality $0a = 0$ follows from the cancellativity of addition and the distributive property. As all the semirings we are interested in here are commutative, from now on we will use the single term \emph{semiring}, assuming without explicit mention the commutativity of both addition and multiplication. A subset $S'$ of a semiring $S$ is a \emph{subsemiring} of~$S$ if $(S',+)$ is a submonoid of $(S,+)$ that contains~$1$ and is closed under multiplication. Note that every subsemiring of $S$ is a semiring.

\begin{definition}
	A semiring $S$ is called a \emph{semidomain} if $S$ is a subsemiring of an integral domain. A semiring $S$ is called a \emph{positive semidomain} if $S$ is a subsemiring of the nonnegative real numbers.
\end{definition}

For the rest of this subsection, assume that $S$ is a semidomain. We let $S^\bullet$ denote the multiplicative monoid $(S \setminus \{0\}, \cdot)$, and we call $S^\bullet$ the \emph{multiplicative monoid} of $S$.  Following terminology and notation from ring theory, we refer to the units of the multiplicative monoid $S^\bullet$ simply as \emph{units} of $S$ and, to avoid ambiguities, we will refer to the units of $(S,+)$ as \emph{invertible elements} of $S$. Also, with the ring theory notation in mind, we let $S^\times$ denote the group of units of $S$, letting $\uu(S)$ denote the additive group of invertible elements of $S$. To simplify notation, we write $\mathcal{A}(S)$ instead of $\mathcal{A}(S^\bullet)$ for the set of irreducibles of the multiplicative monoid $S^\bullet$. The semidomain $S$ is called \emph{bi-atomic} if both monoids $(S,+)$ and $S^\bullet$ are atomic.

To assist in the study of the atomic structure of semidomains, we provide the following strengthening of~\cite[Proposition~3.1]{BCG21}.

\begin{prop}\label{prop:semdom_mult_closed_atom}
    In a positive semidomain $S$, the set of additive atoms is
    multiplicatively divisor-closed.
\end{prop}

\begin{proof}
    Let $a$ be an additive atom of $S$, and suppose for the sake of
    contradiction that there exists a $b \in S$ such that $b \mid_{S^\bullet} a$
    and $b$ is not an additive atom of $S$; let $a = bc$.
    Then we may write $b = x + y$ for some $x, y \in S^\bullet$.
    But then $a = xc + yc$, and since $S$ is positive, $xc$ and $yc$
    are not invertible elements. This contradicts that $a$ is an additive atom.
\end{proof}

It follows that, if there is at least one additive atom in $S$, then
all elements of $S^\times$ are additive atoms, so
that~\cite[Proposition~3.1]{BCG21} follows from
Proposition~\ref{prop:semdom_mult_closed_atom}.

\medskip

\subsection{Pure Irreducibles}
We begin our investigation of pure irreducibility and length-factoriality in semidomains by giving the following generalization of~\cite[Theorem~6.4]{CCGS21}, which places a strong restriction on the pure irreducibles of a semidomain. The proof follows the idea of the proof given in~\cite[Theorem~6.4]{CCGS21}, but with subtraction replaced by addition and some further necessary modifications.

\begin{theorem}\label{thm:sem_no_pls}
    Let $S$ be an atomic semidomain. Then either $\mathcal{L}(S) = \emptyset$ or $\mathcal{S}(S) = \emptyset$.
\end{theorem}

\begin{proof}
    Suppose for the sake of contradiction that $S$ is a semidomain satisfying the PLS property. By~\cite[Corollary~4.4]{CCGS21}, both $\mathcal{L}(S)$ and $\mathcal{S}(S)$ are finite sets. Let $\mathcal{L}(S) = \{ \alpha_i \}_{i = 1}^\ell$ and $\mathcal{S}(S) = \{ \beta_i \}_{i = 1}^s$. By Proposition~\ref{prop:pure_in_unbalanced}, each $\alpha_i$ and $\beta_i$ will appear in any unbalanced factorization relation of $S$, so there exists an irredundant unbalanced factorization relation \begin{equation*}
        \left( z \prod\limits_{i = 1}^\ell \alpha_i^{a_i}, z' \prod\limits_{i = 1}^s \beta_i^{b_i} \right),
    \end{equation*}
    where $a_i$ and $b_j$ are positive integers for all $(i,j) \in \ldb 1,\ell \rdb \times \ldb 1, s \rdb$, while $z$ and $z'$ are arbitrary elements of $\mathsf{Z}(S)$ such that none of the $\alpha_i$ or $\beta_i$ appear in $z$ or $z'$. We now consider the following three cases.
    
    \textit{Case 1:} $\ell, s \geq 2$. We consider the element \begin{equation*}
        x_1 := \pi(z) (\alpha_1^{a_1} + \beta_1^{b_1}) \prod\limits_{i = 2}^\ell \alpha_i^{a_i}
    \end{equation*}
    of $S$. To check that $x_1$ is nonzero, it suffices to check that $\alpha_1^{a_1} + \beta_1^{b_1}$ is nonzero. Suppose that $\alpha_1^{a_1} + \beta_1^{b_1} = 0$. Then $\alpha_1^{2a_1} = \beta_1^{2b_1}$, so $(\alpha_1^{2a_1}, \beta_1^{2b_1})$ is an unbalanced factorization relation of $S$. This forces $\alpha_2$ to appear in the left component, which is a contradiction. Therefore $x_1$ is nonzero.
    
    By definition, $\beta_1$ divides $x_1$. Therefore there exist $w_1 \in \mathsf{Z}(S)$ and $w_2 \in \mathsf{Z}(\alpha_1^{a_1} + \beta_1^{b_1})$ such that $\beta_1 w_1 \in \mathsf{Z} (x_1)$ and $w := w_2 z \prod_{i = 2}^\ell \alpha_i^{a_i} \in \mathsf{Z}(x_1)$. We now consider the factorization relation $(\beta_1 w_1, w)$.
    
    \textit{Case 1.1:} $\beta_1$ does not appear in $w$. Because $\beta_1 \in \mathcal{S}(S)$, the factorization relation is unbalanced, and so $\alpha_1$ must appear in $w$, and therefore in $w_2$. But this implies that $\alpha_1$ divides $\beta_1^{b_1}$, giving a factorization relation $(\alpha_1 w_1', \beta_1^{b_1})$ for some $w_1' \in \mathsf{Z}(S)$. Since $\alpha_1$ is present in the left component, but not the right component, this is an unbalanced factorization relation. But $\beta_2$ is not present in the right component, which contradicts that $\beta_2 \in \mathcal{S}(S)$.
    
    \textit{Case 1.2}: $\beta_1$ appears in $w$. Because $\beta_1$ does not appear in $z$, it must appear in $w_2$. So $\beta_1$ divides $\alpha_1^{a_1}$, giving a contradiction analogous to the one obtained in Case 1.1.
    
    \textit{Case 2:} $\{\ell, s \} = \{ 1, n \}$ for some $n \geq 2$. We assume that $\ell = 1$; the other case is similar. We impose the condition that $a_1$ is the minimal exponent of $\alpha_1$ among all unbalanced factorization relations. We consider the element \begin{equation*}
        x_2 := \pi(z) \alpha_1^{a_1 - 1} (\alpha_1 + \beta_1^{b_1})
    \end{equation*}
    of $S$. As in Case 1, the element $x_2$ must be nonzero, as otherwise $(\alpha_1^2, \beta_1^{2b_1})$ is an unbalanced factorization relation in which not every pure irreducible appears. Furthermore, we have $\beta_1 \mid x_2$ by definition, because $\ell = 1$.
    
    As in Case 1, there exist $w_1 \in \mathsf{Z}(S)$ and $w_2 \in \mathsf{Z}(\alpha_1 + \beta_1^{b_1})$ giving a factorization relation $(\beta_1 w_1, z \alpha_1^{a_1 - 1} w_2)$. Because $\beta_1$ does not divide $\alpha_1 + \beta_1^{b_1}$, it does not appear in $z \alpha_1^{a_1 - 1} w_2$, so the factorization relation is unbalanced. By the minimality of $a_1$, the irreducible $\alpha_1$ must appear in $w_2$. But then $\alpha_1$ divides $\beta_1^{b_1}$, giving a contradiction as in Case 1.1.
    
    \textit{Case 3:} $\ell = s = 1$. We again impose the minimality condition of Case 2 on $a_1$. We consider the element \begin{equation*}
        x_3 := \pi(z) \alpha_1^{a_1 - 1} (\alpha_1 + \beta_1)
    \end{equation*}
    of $S$. If $\alpha_1 + \beta_1 = 0$, then $(\alpha_1^2, \beta_1^2)$ is a balanced factorization relation involving pure irreducibles, which is a contradiction. So $x_3$ is nonzero.
    
    As in the previous cases, there exist $w_1 \in \mathsf{Z}(S)$ and $w_2 \in \mathsf{Z} (\alpha_1 + \beta_1)$ giving a factorization relation $(\beta_1 w_1, z \alpha_1^{a_1 - 1} w_2)$. Because $\beta_1$ does not divide $\alpha_1 + \beta_1$, it does not appear in $z \alpha_1^{a_1 - 1} w_2$, so the factorization relation is unbalanced. By the minimality of $a_1$, the irreducible $\alpha_1$ must appear in $w_2$. But this is a contradiction, as $\alpha_1$ does not divide $\alpha_1 + \beta_1$.
\end{proof}

Because every proper LFM is a PLSM by~\cite[Corollary~4.6]{CCGS21}, we obtain the following corollary.

\begin{cor}
    The multiplicative monoid of a semidomain $S$ is an LFM if and only if it is a UFM.
    \label{thm:lfs_ufs}
\end{cor}

\medskip

\subsection{Restrictions on Bi-LFSs}
In this section, we discuss an open problem posed in~\cite{BCG21}. To do so, we need the following definition. The semidomain $S$ is called a \emph{bi-LFS} (resp., \emph{bi-UFS}) provided that both $(S,+)$ and $S^\bullet$ are LFMs (resp., UFMs). The properties of being bi-LFS and being bi-UFS have been studied recently by Baeth et al. in~\cite{BCG21} and by Gotti and Polo in~\cite{GP22}. Observe that $\nn_0$ is a bi-UFS and, therefore, a bi-LFS. It turns out that this is the only positive semidomain that is known to be a bi-LFS. This brings us to the following question, first stated in~\cite{BCG21}. 

\begin{question} \cite[Question~7.8]{BCG21} \label{quest:main question}
	Is $\nn_0$ the only positive subsemidomain of $\rr$ that is a bi-LFS?
\end{question}

Corollary~\ref{thm:lfs_ufs} suggests the following conjecture, which, if proved, greatly restricts the possibilities for counterexamples to Question~\ref{quest:main question}.

\begin{conj}
    A positive semidomain is a bi-LFS if and only if it is a bi-UFS.
    \label{conj:lfs_ufs}
\end{conj}

The following construction, called \emph{exponentiation of positive monoids}, seems to have first been used in~\cite{BG20}. The semidomain resulting from that construction has been crucial to give desired counterexamples to understand the atomic structure of semidomains (see \cite{BCG21} and \cite{GP22}).
\medskip

\noindent \textbf{Exponentiation of Positive Monoids:} Given a positive monoid $M$, we can use the Lindemann-Weierstrass Theorem from transcendental number theory (see \cite[Chapter~1]{aB90}) to create a positive semidomain $S$ of the real line that has a copy of $M$ as a multiplicative divisor-closed submonoid. The Lindemann-Weierstrass Theorem states that if $\alpha_1, \dots, \alpha_n$ are distinct algebraic numbers, then the set $\{e^{\alpha_1}, \dots, e^{\alpha_n} \}$ is linearly independent over the algebraic numbers. In particular, if $M$ is a positive monoid consisting of algebraic numbers, then the additive monoid
\[
	E(M) := \langle e^m \mid m \in M \rangle
\]
is free on the set $\{e^m \mid m \in M\}$. In addition, one can readily see that $E(M)$ is closed under the standard multiplication and, therefore, is a positive semidomain. Observe that the multiplicative submonoid
\[
	e(M) := \{e^m \mid m \in M\}
\]
of $E(M)^\bullet$ is isomorphic to the monoid~$M$. It turns out that $e(M)$ is divisor-closed in $E(M)^\bullet$ \cite[Lemma~2.7]{BCG21} and, therefore, atoms and factorizations in the smaller monoid $e(M)$ are preserved in the larger monoid $E(M)^\bullet$.
\smallskip

Puiseux monoids, in tandem with the exponentiation construction, have been a recurrent source of counterexamples in the context of positive semidomains (see \cite{BCG21,BG20,GP22}). We would like to finish this subsection by emphasizing that any potential counterexample answering Question~\ref{quest:main question} formed by exponentiating a Puiseux monoid will have precisely two generators.

\begin{prop}\label{prop:exp_bi_lfs}
	Let $M$ be a Puiseux monoid. If $E(M)$ is a bi-LFS, then $M \cong \langle a,b \rangle$, where $a$ and $b$ are distinct positive integers with $\gcd(a,b) = 1$.
\end{prop}

\begin{proof}
	Set $S = E(M)$. Since $S$ is the result of the exponentiation construction, the positive monoid $(S,+)$ is a UFM. Suppose for the sake of contradiction that $M$ is not isomorphic to a numerical semigroup $\langle a, b \rangle$ for any distinct $a,b \in \nn$ with $\gcd(a,b) = 1$. We split the rest of the proof into the following two cases.
	\smallskip
	
	\noindent {\it Case 1:} $M$ is isomorphic to $\nn_0$. Then we can set $M = \nn_0$. In this case, the semidomains $S$ and $\nn_0[X]$ are isomorphic, and so are the multiplicative monoids $S^\bullet$ and $\nn_0[X]^\bullet$. Now notice that the polynomials $X^3 + 2X + 3$ and $X^3 + X^2 + X + 6$ are irreducible in $\nn_0[X]^\bullet$ as their unique factorizations in the UFD $\zz[X]$ are $X^3+2X+3 = (X+1)(X^2-X+3)$ and $X^3 + X^2 + X + 6 = (X+2)(X^2-X+3)$, respectively. Hence the elements $e+1$, $e+2$, $e^3+e^2+e+6$, and $e^3+2e+3$ of $S$ are irreducibles in $S^\bullet$. Also, the fact that $S$ is reduced implies that such irreducibles are pairwise non-associates. As a result, the equality $(e+1) (e^3+e^2+e+6) = (e+2) (e^3+2e+3)$ contradicts the fact that $S$ is a bi-LFS.
	\smallskip
	
	\noindent {\it Case 2:} $M$ is not isomorphic to $\nn_0$.  Since $M$ is isomorphic to a multiplicative divisor-closed submonoid of $S^\bullet$, the fact that $S$ is a bi-LFS ensures that $M$ is atomic. Since $M$ is not isomorphic to $\nn_0$, we see that $|\mathcal{A}(M)| \ge 2$. If $|\mathcal{A}(M)| = 2$, then $M$ is isomorphic to a numerical monoid $\langle a,b \rangle$, which is not possible by assumption. Hence $|\mathcal{A}(M)| \ge 3$. Then \cite[Proposition~4.3]{fG20b} guarantees that $M$ is not an LFM, and so the divisor-closed submonoid $e(M)$ of $S^\bullet$ is not a LFM. However, this contradicts that $S$ is a bi-LFS.
\end{proof}
\bigskip

\medskip
\section*{Acknowledgments}

We thank our mentor, Dr.~Felix Gotti, for guiding our research project and making valuable suggestions, and the MIT PRIMES organizers for making our research possible.
We would also like to thank the Reviewer for their thoughtful comments and suggestions.

\bigskip

\end{document}